\documentclass[oneside,12pt]{amsart}
\usepackage{amssymb, mathrsfs, amscd}
\usepackage[all]{xy}
\usepackage{enumerate}

\newtheorem*{thma}{Theorem A}
\newtheorem*{thmb}{Theorem B}
\newtheorem*{thmc}{Theorem C}
\newtheorem*{thmd}{Theorem D}

\newtheorem{theorem}{Theorem}[section]
\newtheorem{definition}[theorem]{Definition}
\newtheorem{lemma}[theorem]{Lemma}

\newtheorem{corollary}[theorem]{Corollary}

\newtheorem{remark}[theorem]{Remark}

\newcommand{\Z}{{\mathbb Z}}

\newcommand{\R}{{\mathbb R}}
\newcommand{\C}{{\mathbb C}}
\newcommand{\lieu}{{\mathfrak u}}
\newcommand{\lieg}{{\mathfrak g}}
\newcommand{\lien}{{\mathfrak n}}
\newcommand{\liev}{{\mathfrak v}}
\newcommand{\liew}{{\mathfrak w}}
\newcommand{\Aut}{\mbox{Aut}}
\newcommand{\GL}{\mbox{GL}}

\author{Karel Dekimpe*}
\address{Dept. of Math., Katholieke Universiteit Leuven, Campus Kortrijk, Kortrijk, Belgium}%
\email{Karel.Dekimpe@kuleuven-kortrijk.be}%

\author{Nansen Petrosyan**}
\address{Dept. of Math., Katholieke Universiteit Leuven, Campus Kortrijk, Kortrijk, Belgium}%
\email{Nansen.Petrosyan@kuleuven-kortrijk.be}%
\title[crystallographic actions on algebraic manifolds]{Crystallographic actions on contractible algebraic manifolds}
\thanks{*Supported by the Research Fund K.U.Leuven.}%
\thanks{**Supported by the Research Fund K.U.Leuven and FWO-Flanders Research Fellowship.}%
\date{\today}
\begin{document}
\begin{abstract}
We study properly discontinuous and cocompact actions of a discrete subgroup $\Gamma$ of an algebraic group $G$ on a contractible algebraic manifold $X$. We suppose that this action comes from an algebraic action of $G$  on $X$ such that a maximal reductive subgroup of $G$ fixes a point. When the real rank of any simple subgroup of $G$ is at most one or the dimension of $X$ is at most three, we show that $\Gamma$ is virtually polycyclic. When $\Gamma$ is virtually polycyclic, we show that the action reduces to a NIL-affine crystallographic action. Specializing to NIL-affine actions, we prove that the generalized Auslander conjecture  holds up to dimension six and give a new proof of the fact that every virtually polycyclic group admits a NIL-affine crystallographic action.
\end{abstract}
\maketitle

\bigskip

\section{Introduction}

The study of crystallographic groups and actions already has a long history.
The original concept was that of the Euclidean crystallographic groups. These are discrete and cocompact subgroups
of the group of isometries of a Euclidean space. The crystallographic groups  are then exactly those groups
acting properly discontinuously, cocompactly and isometrically on a Euclidean space.
These  groups are well understood by the three Bieberbach theorems derived almost a century ago (see \cite{bieb1}, \cite{bieb2}).
The first Bieberbach theorem describes the algebraic structure of such crystallographic groups and implies that
all of them are virtually abelian. The torsion-free crystallographic groups, called Bieberbach groups, are exactly the
fundamental groups of the compact flat Riemannian manifolds.

About half a century later, L.~Auslander in his fundamental paper \cite{ausl1} generalized the
first two Bieberbach theorems to the class of almost crystallographic groups. These are subgroups of Isom$(N)$, the group of isometries
of a 1-connected nilpotent Lie group $N$ equipped with a left--invariant Riemannian metric, acting
properly discontinuously and cocompactly on $N$. The generalization of the first theorem implies that any almost crystallographic subgroup is
virtually nilpotent. The torsion-free almost crystallographic groups are then the fundamental groups of the almost flat manifolds in
the sense of Gromov (see \cite{Gromov}, \cite{Ruh}). Although this is of no influence to the sequel of our paper, we note that
the generalization of the second Bieberbach theorem in \cite{ausl1} is not correct. See \cite{LR} for the right generalization.

Around the same time and in connection with the study of affine manifolds, one also started to deal with affine crystallographic groups. These are subgroups of Aff$(\R^n)$, the group of invertible affine motions of a Euclidean $n$-space, acting properly discontinuously and cocompactly on $\R^n$.
For this class of affine crystallographic groups, an analogue of the first Bieberbach theorem has not yet been proven. In fact, two main
problems, which are more or less converses of each other, have been dominating much of the research in this field. The first problem is a question posed
by J.~Milnor in 1977 (see \cite{miln}), who asked whether or not it was true that any torsion-free polycyclic-by-finite group could be realized as
an affine crystallographic group.  On the other hand, there is a conjecture due to L.~Auslander (see \cite{ausl2}) stating that any affine crystallographic group
is polycyclic-by-finite. In fact L.~Auslander formulated this as a theorem, but it has been shown that his proof contains a gap and since then
it is referred to  as Auslander's conjecture. It is clear that a positive answer to both Milnor's question and Auslander's conjecture,
could be interpreted as a first Bieberbach theorem for affine crystallographic groups (giving a clear description of their algebraic structure).
Unfortunately, the answer to Milnor's question is negative, even for nilpotent groups, as was shown by Y.~Benoist in \cite{Benoist}.
Moreover, the Auslander conjecture is still open and is only known to hold up to dimension 6. (In fact, up to our knowledge,
a full proof has only been given up to dimension 3 (see \cite{FG}) and a positive result up to dimension 6 has been announced in \cite{AMS}).

Inspired by the negative answer to Milnor's question, one has been looking for other possible generalizations for which the analogue of
Milnor's question has a positive answer. The first positive result in this direction was given in \cite{DI}. It  was shown that
any polycyclic-by-finite group admits a properly discontinuous and cocompact action on some $\R^n$ where the action is given by
polynomial maps of bounded degree. So, although not all torsion-free polycyclic-by-finite groups can be realized as affine crystallographic groups, they
can be realized as  polynomial crystallographic groups.

To introduce yet another generalization, we again consider a 1-connected nilpotent Lie group $N$, and
define the affine group of $N$ to be the group $\mbox{Aff}(N)=N\rtimes \mbox{Aut}(N)$ acting on $N$ by $(m,h)\cdot n= m h(n)$, for all
$m,n\in N$ and all $h\in \mbox{Aut}(N)$. A subgroup $\Gamma\subseteq \mbox{Aff}(N)$ acting properly discontinuously and
cocompactly on $N$ is called a NIL--affine crystallographic group. For $N=\R^n$, this immediately reduces to the
ordinary case of affine crystallographic groups.  It has been shown in \cite{dek} and independently in \cite{baues} that any
torsion free polycyclic-by-finite group can also be realized as a NIL-affine crystallographic group, so also in this situation the
analogue of Milnor's question has a positive answer.

On the other hand, for all of these possible generalizations one can also study a generalized version of Auslander's conjecture.
Some results in this direction for polynomial and NIL-affine actions can be found in \cite{dek1, BDD}

However, up till now, all approaches to the Auslander conjecture and its generalizations have been
dealing with the specific situation (affine, polynomial or NIL-affine).
In this paper we introduce a unified approach to all known generalizations of the concept of a crystallographic group and start
the study of the generalized version of the Auslander conjecture in this most general setting.

\medskip

The general setting is obtained by replacing $\R^n$ or the 1-connected nilpotent Lie group $N$ (i.e.~the space on which the
crystallographic groups are acting) by a contractible real algebraic manifold $X$ and by considering subgroups $\Gamma\subseteq G$, or slightly more general
representations $\rho:\Gamma\rightarrow G$, where $G$ is a real algebraic group acting algebraically on $X$.

\begin{definition}\label{def-cryst}{\normalfont
Let $G$ be a real algebraic group acting algebraically on a contractible real algebraic manifold $X$.
A representation  $\rho:\Gamma\rightarrow G$ letting $\Gamma$ act properly discontinuously and cocompactly on $X$ is an
{\it algebraic crystallographic action}. The image $\rho(\Gamma)$ is said to be an {\it algebraic crystallographic group}.}
\end{definition}
In case there is no confusion possible, we will just talk about crystallographic actions and crystallographic groups.

In the examples above, the groups $\mbox{Isom}(\R^n)$, $\mbox{Isom}(N)$, $\mbox{Aff}(\R^n)$ and $\mbox{Aff}(N)$ are indeed
algebraic groups acting algebraically on $\R^n$ (or
$N$ which is isomorphic to $\R^n$ as an algebraic manifold). Also, in the case of crystallographic groups of polynomial actions of bounded degree, the
action factorizes through the action of a real algebraic group, as it is explained in \cite{BD}. It follows that all types of crystallographic groups considered thus far are
also algebraic crystallographic groups.

In all of the above types of crystallographic groups, a key observation is that a maximal reductive subgroup of the algebraic closure of $\Gamma$ has a fixed point.  This leads us to the following generalization of Auslander's conjecture.

\medskip

\parbox[t]{6in}{{\it Question 1}: If a maximal reductive subgroup of $G$ fixes a point of $X$ and $\Gamma$ acts crystallographically on $X$ via a
representation $\rho:\Gamma\rightarrow G$, is $\Gamma$ virtually polycyclic?}

\medskip

Note that a positive answer to this question would provide a full geometric characterization of the class of polycyclic-by-finite groups.
Indeed, we already know that any such group admits an algebraic crystallographic action (e.g.~a NIL-affine crystallographic), and a positive
answer would imply that these are the only groups admitting a crystallographic action with the extra condition that
a maximal reductive subgroup of $G$ has fixed point.

\medskip

\noindent Using results of Section 2, we can equivalently ask a more specific:

\medskip

\parbox[t]{6in}{{\it Question 2}: Let $\Gamma$ be Zariski dense in a connected algebraic group $G$. Let $R$ be the radical, $U$ be the unipotent radical, and $H$ be a maximal reductive subgroup of $G$. Let $V$ be a closed subgroup of $U$ normalized by $H$. Define a left action of $G$ on the homogeneous space $U/V$ of left cosets by $$wh\cdot[u] = [whuh^{-1}], \hspace{2mm}\forall u,w\in U, \forall h\in H.$$ If $\Gamma\cap R=\{e\}$ and the action of $\Gamma$ on $U/V$ is crystallographic, is $\Gamma$ trivial?}

\medskip

Expanding on the results of Soifer and Tomanov  on affine actions by  generalized Lorentz motions (see \cite[Theorem A]{Soifer}, \cite[Theorem A]{Tomanov}), we derive the following theorem,
which can be seen as a first positive indication for the general Auslander conjecture.

\begin{thma}[Theorem \ref{thm: main}] \label{thm: main_intro} {\normalfont Let $\Gamma$ be a  subgroup of an algebraic group $G$. Let $G=UH$ where $U$ is the unipotent radical and $H$ is a maximal reductive subgroup. Suppose $G$ acts algebraically on a contractible algebraic manifold $X$ such that there exists $x_0\in X$ with $Hx_0=x_0$. Suppose $\Gamma$ acts crystallographically on $X$.
\begin{enumerate}[(a)]
\item If the real rank of any simple subgroup of $G$ does not exceed one, then $\Gamma$ is virtually polycyclic.
\smallskip
\item If $\Gamma$ is a virtually polycyclic, Zariski dense subgroup of $G$, then $X$ is $G$-equivariantly isomorphic to $U$ with the $G$-action  given by
$$wh\cdot u = whuh^{-1}, \hspace{2mm}\forall u,w\in U, \forall h\in H.$$
\end{enumerate}}
\end{thma}

The second part of the theorem may be interpreted by deducing that every such action must necessarily be NIL-affine crystallographic and in this sense rigid (see Remark \ref{rmk: aftermain}). To illustrate the practicality of this result, for an arbitrary virtually polycyclic group $\Gamma$, we construct a crystallographic action on a contractible algebraic manifold and use this rigidity  to obtain that $\Gamma$ admits a NIL-affine crystallographic action (see Theorem \ref{thm: action}). In this way, we give a new proof of the aforementioned result appearing in \cite{dek} and  in \cite{baues}.

Next, we specialize to NIL-affine crystallographic actions and show that the generalized Auslander conjecture remains valid
up to dimension six.

\begin{thmb}[Theorem \ref{thm: app1}] \label{thm: app1_intro} {\normalfont Let $\Gamma$ be a NIL-affine crystallographic group of a 1-connected nilpotent Lie group $N$. If the real rank of any simple subgroup of the algebraic closure of $\Gamma$ does not exceed one, then $\Gamma$ is virtually polycyclic and  $N$ is isomorphic to the unipotent radical of the algebraic closure of $\Gamma$.}
\end{thmb}

\begin{thmc}[Theorem \ref{thm: app2}] \label{thm: app2_intro} {\normalfont Suppose $\Gamma$ is a NIL-affine crystallographic group of a 1-connected nilpotent Lie group $N$ of dimension at most six. Then $\Gamma$ is virtually polycyclic and  $N$ is isomorphic to the unipotent radical of the algebraic closure of $\Gamma$.}
\end{thmc}

At last, we consider groups that  admit properly discontinuous actions on a contractible algebraic manifold of dimension at most three. Consequently, we answer affirmatively  Question 1 up to this dimension.
\begin{thmd}[Theorem \ref{thm: app3}] {\normalfont Let $\Gamma$ be a finitely generated Zariski dense subgroup of a connected algebraic group $G$. Let $H$ be a maximal reductive subgroup of $G$. Suppose $G$ acts algebraically on a contractible algebraic manifold $X$ such that there exists $x_0\in X$ with $Hx_0=x_0$. Suppose $\Gamma$ acts properly discontinuously on $X$.
\begin{enumerate}[(a)]
\item If $\mbox{dim}(X)=2$, then $\Gamma$ is virtually polycyclic.
\smallskip
\item If  $\mbox{dim}(X)=3$ and $\Gamma$ is not virtually polycyclic, then $\Gamma$ is virtually free, $X$ is $G$-equivariantly isomorphic to $\R^3$, where $G$ acts by Lorentz transformations on $\R^3$.
\end{enumerate}}
\end{thmd}
For examples and a survey of actions by Lorentz transformations, we refer the reader to \cite{DG_1}, \cite{DG_2}, and \cite{CDGM}.

\section{Crystallographic Actions}

In this section, we derive the necessary lemmas for our main results. We say that $X$ is an {\em algebraic manifold} if it is a nonsingular real affine variety (see \cite{BCR}, \cite{Shaf}). By an isomorphism of algebraic manifolds, we mean an isomorphism (biregular map) of the affine varieties. We say that an algebraic group $G$ acts {\em algebraically} on $X$ if the action map $G\times X\rightarrow X$ is a morphism of algebraic varieties. By an algebraic group, we will always mean a real algebraic group, which in turn will be the set of real points of a linear algebraic group.

In general setting, $\Gamma$ will be a subgroup of a real linear algebraic group $G$ that acts algebraically on a contractible algebraic manifold $X$ such that the restriction of the action to $\Gamma$ is properly discontinuous and cocompact. As already stated in Definition~\ref{def-cryst}, we will call such an action by $\Gamma$ {\em crystallographic}.

Note that, since $G$ acts continuously and  $\Gamma$ acts properly discontinuously on $X$, $\Gamma$ must necessarily be discrete in $G$. Of course, the notions
``properly discontinuous'' and ``discrete'' refer to the usual topology and not to the  Zariski topology.

The first few lemmas we present contain some small but indispensable results related to properly discontinuous and cocompact actions.

\begin{lemma}\label{lem: lem1} {\normalfont Suppose $G$ is a group acting on a locally compact Hausdorff space $X$. Let $\pi: X\rightarrow X/G$ denote the natural quotient map. Then, for any closed compact subset $K\subseteq X/G$, there exists a compact set $K'\subseteq X$ such that $\pi(K')=K$.}
\end{lemma}

\begin{proof} Let $S$ be the preimage of $K$. Since $K$ is closed, $S$ is also closed. For each $x\in S$, let $U_x$ be a neighborhood of $x$ in $X$ such that $\overline{U}_x$ is compact.  As $\{U_x \;|\; x\in S\}$ covers $S$, $\{\pi(U_x) \;|\; x\in S\}$ is an open cover of $K$. So, we can find a finite
 subcover $\{\pi(U_{x_i}) \;|\; x_i \in S, i\in I\}$ of $K$. Let $K'= \bigcup_{i\in I} \overline{U}_{x_i}\cap S.$ Then $K'$ is a closed subset of a compact set and it is therefore compact. Clearly, $\pi(K')=K$.
\end{proof}

\begin{lemma} \label{lem: lem2}{\normalfont Let $\Gamma$ be a group acting properly discontinuously and cocompactly on a connected locally compact Hausdorff space $X$. Then $\Gamma$ is finitely generated.}
\end{lemma}

\begin{proof} By Lemma \ref{lem: lem1}, we can find a compact subset $C$ of $X$ such that $\Gamma C=X$. Let $U$ be open subset containing $C$ such that $\overline{U}$ is compact. Then, the set $\mathscr{F}=\{\gamma\in \Gamma \;|\; \gamma U \cap U\ne \emptyset\}$ is finite. We claim that $\Gamma=\langle \mathscr{F}\rangle$.

Let $\Lambda=\langle \mathscr{F}\rangle$ and set $A=\bigcup_{\gamma\in\Lambda} \gamma U$ and $B=\bigcup_{\gamma\in\Gamma \smallsetminus\Lambda} \gamma U$. Since $A$ and $B$ are open and $A\cup B=X$, it follows that either $A\cap B\ne\emptyset$ or $B=\emptyset$. If $B\ne \emptyset$, then there exists $\gamma'\in \Gamma \smallsetminus\Lambda$  for which $\gamma'U\cap A\ne \emptyset$. Hence, there exists $\gamma \in \Lambda$ such that $\gamma'U\cap \gamma U\ne \emptyset$. This implies $\gamma^{-1}\gamma'U\cap U\ne \emptyset$, showing that $\gamma^{-1}\gamma' \in \Lambda$. But then $\gamma'\in \Lambda$, which is a contradiction. Thus, $B=\emptyset$ and $\Gamma= \Lambda$.
\end{proof}

The following is a well-known fact on aspherical manifolds (see 8.1 of \cite{brown}).

\begin{lemma}\label{lem: lem3} {\normalfont Let $\Gamma$ be a group acting freely and properly discontinuously on an $n$-dimensional, contractible topological manifold $M$ without boundary. Then $cd(\Gamma)\leq n$ with equality if and only if $M/\Gamma$ is compact.}
\end{lemma}

In our study of algebraic actions, we will often deal with actions of unipotent groups and their homogeneous spaces. In the next lemma, we collect some basic results in this context.

\begin{lemma}\label{lem: lem4} {\normalfont Suppose $U$ is a connected, unipotent group.
\begin{enumerate}[(a)]
\item If $V$ is a closed, connected subgroup of $U$, then $U/V$ is an algebraic manifold isomorphic to $\mathbb R^k$ where $k=\mbox{dim}(U)-\mbox{dim}(V)$.
\smallskip
\item Suppose $U$ acts algebraically on an algebraic manifold $X$. Then for each $x\in X$, the orbit $Ux$ is Zariski closed in $X$. Also, the isotropy group $U_x$ of $x$ is connected subgroup of $U$.
\smallskip
\item With the assumptions of part (b), for each $x\in X$, $Ux$ is isomorphic to $\mathbb R^k$ where $k=\mbox{dim}(U)-\mbox{dim}(U_x)$.
\smallskip
\end{enumerate}}
\end{lemma}

\begin{proof} Let $\{X_1, \dots , X_n\}$ be a weak Mal'cev basis of the Lie algebra $\lieu$ through the subalgebra $\liev$ where $U=\mbox{exp}(\lieu)$ and $V=\mbox{exp}(\liev)$.
Then the map $\phi:\R^{k}\rightarrow U/V$ given by $$(t_1, \dots , t_{k})\mapsto  \mbox{exp}(t_1 X_{n-k+1})\dots\mbox{exp}(t_{k} X_n)\cdot V$$ is an isomorphism of algebraic manifolds (see Theorem 1.2.12 and proceeding Remark 1 of \cite{CG}). This proves (a).

The fact that the orbit $Ux$ is  Zariski closed in $X$ is  now standard and was first proven independently  by Kostant and Rosenlicht (see \cite{Ros}). For the second part of (b), we let $u\in U_x$ and note that the one parameter subgroup $u^t$ is the minimal algebraic subgroup containing $u$. Therefore, $u^t\leq U_x$, proving that $U_x$ is connected.

Part (c) follows directly from (a) and (b).
\end{proof}

The next lemma is a generalization of a result first observed by Margulis on affine crystallographic actions (see Prop.1 of \cite{Soifer}).

\begin{lemma}\label{lem: lem5} {\normalfont Let $\Gamma$ be a subgroup of an algebraic group $G$. Let $G=UH$ where $U$ is the unipotent radical and $H$ is a reductive subgroup. Suppose $G$ acts algebraically on a contractible algebraic manifold $X$ such that there exists $x_0\in X$ with $Hx_0=x_0$. If $\Gamma$ acts crystallographically on $X$, then $U$ acts transitively on $X$.}
\end{lemma}

\begin{proof}Since $\Gamma$ acts crystallographically, by Lemma \ref{lem: lem2}, it is finitely generated. By Selberg's Lemma, we can find a torsion-free subgroup $\Gamma'$ of finite index in $\Gamma$. By Lemma \ref{lem: lem3}, $\mbox{cd}(\Gamma')=\mbox{dim}(X)$, implying $\mbox{vcd}(\Gamma)=\mbox{dim}(X)$.

On the other hand, $Ux_0=Gx_0$ is $\Gamma$-invariant. By Lemma \ref{lem: lem4}, it is a closed and contractible submanifold of $X$. This implies $\mbox{vcd}(\Gamma)\leq \mbox{dim}(Ux_0)$, resulting in $\mbox{dim}(Ux_0)=\mbox{dim}(X)$ and hence $Ux_0=X$.
\end{proof}

In our study of algebraic crystallographic groups, we will often be able to reduce the situation, by first studying the radical of the algebraic group $G$. With this in mind, next we will derive a few necessary results in case $G$ is solvable.

\begin{lemma}\label{lem: lem6.5}{\normalfont Let $\Gamma_1$ be a Zariski dense, solvable subgroup of an algebraic group $G_1$. Let $G_1=U_1H_1$ where $U_1$ is the unipotent radical and $H_1$ is a reductive subgroup. Then, there exists a compact subset $K\subset U_1$ such that $G_1=\Gamma_1KH_1$.}
\end{lemma}

\begin{proof} Since $G_1$ has a finite number of connected components, we can assume it is connected. As a general fact, the commutator subgroup $[\Gamma_1, \Gamma_1]$ is Zariski dense in $[G_1, G_1]$ (see p.59 of \cite{Borel}). Let $G_2= [G_1, G_1]$ and $\Gamma_2=G_2\cap \Gamma_1$. Then $[\Gamma_1, \Gamma_1]\leq \Gamma_2$. So, $\Gamma_2$ is Zariski dense in $G_2$. Let $\widetilde{G}_1=G_1/G_2$ and $\widetilde{\Gamma}_1=\Gamma_1/\Gamma_2$. Since $G_2$ is unipotent, we can find a compact subset $K_2$ such that $G_2=\Gamma_2K_2$ (see 2.3 of \cite{Rag}). Since $\widetilde{G}_1$ is abelian, there is a natural epimorphism $\psi: \widetilde{G}_1\rightarrow \widetilde{U}_1$ onto the unipotent radical $\widetilde{U}_1$. It follows that the image $\psi(\widetilde{\Gamma}_1)$ is Zariski dense inside $\widetilde{U}_1$ (see p.57 of \cite{Borel}). So, there exists a compact subset $\widetilde{K}_1$ in $\widetilde{U}_1$ such that $\widetilde{U}_1=\psi(\widetilde{\Gamma}_1)\widetilde{K}_1$. This shows that $\widetilde{G}_1=\widetilde{\Gamma}_1\widetilde{K}_1\widetilde{H}_1$ where $\widetilde{H}_1$ is the maximal reductive subgroup of $\widetilde{G}_1$.

 Let $G_1=U_1H_1$ where $U_1$ is the unipotent radical and $H_1$ is a reductive subgroup. Denote by $\phi:G_1\rightarrow \widetilde{G}_1$ the natural quotient homomorphism and note that $\phi(U_1)=\widetilde{U}_1$ and $\phi(H_1)=\widetilde{H}_1$. Let $K_1$ be a compact subset of $U_1$ such that $\phi(K_1)=\widetilde{K}_1$. Then, we have $G_1=\Gamma_1G_2K_1H_1=\Gamma_1K_2K_1H_1$ and $K_2K_1 \subseteq U_1$. This finishes the proof.
\end{proof}

\begin{lemma}\label{lem: lem7} {\normalfont Let $\Gamma_1$ be a Zariski dense, solvable subgroup of an algebraic group $G_1$. Let $U_1$ be the unipotent radical and $H_1$ be a maximal reductive subgroup of $G_1$. Suppose $G_1$ acts algebraically on a contractible algebraic manifold $X$ such that there exists $x_0\in X$ with $H_1x_0=x_0$ and suppose the restriction of the action to $\Gamma_1$ is properly discontinuous. Then, $\Gamma_1$ acts crystallographically on the orbit space $G_1x_0$ and $U_1$ acts freely and transitively on $G_1x_0$.}
\end{lemma}

\begin{proof} The fact that $\Gamma_1$ acts crystallographically follows directly  from \ref{lem: lem6.5}. Transitivity of the action of $U_1$ is clear. Next, we show that this action is also free, i.e. $U_1$ acts freely on $U_1x_0$.

In fact, it is enough to show that for each nontrivial $u\in U_1$, $ux_0\ne x_0$. Suppose otherwise, that an element $u$ in $U_1$ fixes $x_0$. By Lemma \ref{lem: lem6.5}, we have that for each integer $i$, $u^i=\gamma_ic_ih_i$ where $\gamma_i\in \Gamma_1$, $c_i\in K$, and $h_i\in H_1$. Then, $u^ix_0=\gamma_ic_ix_0$. Suppose the set $\mathscr{A}=\{\gamma_i\;|\;i\in \Z\}$ is finite. Then, there exists a infinite subsequence $\{u^{i_j}\}$ such that $u^{i_j}=\gamma c_{i_j}h_{i_j}$ where $\gamma\in \mathscr{A}$,  $c_{i_j}\in K$, and $h_{i_j}\in H_1$. Let $\gamma=\gamma_u\gamma_h$ where $\gamma_u\in U_1$ and $\gamma_h\in H_1$. Since $u^{i_j}=\gamma_u\gamma_h c_{i_j}\gamma^{-1}_h\gamma_hh_{i_j}$, it follows that $\gamma_hh_{i_j}=1$ for all $j$. Then $u^{i_j}=\gamma_u\gamma_h c_{i_j}\gamma^{-1}_h$, showing that the set $\{u^{i_j}\;|\;j\in \Z\}$ is contained in the compact set $\{\gamma_u\}K^{\gamma_h}$, which is clearly a contradiction. Hence, $\mathscr{A}$ is infinite. Now, we have $x_0=u^ix_0=\gamma_ic_ix_0$ for all $i$. This implies that the set $$\{\gamma\in \Gamma_1\;|\; Kx_0\cap \gamma Kx_0\ne \emptyset\}$$ is infinite, which is again a contradiction.
\end{proof}

\begin{lemma}\label{lem: lem5.5} {\normalfont Let $\Gamma_1$ be a discrete, finitely generated, Zariski dense subgroup of a connected solvable algebraic group $G_1=U_1H_1$ where $U_1$ is the unipotent radical and $H_1$ is a  reductive subgroup.  Suppose $G_1$ is a normal subgroup of an algebraic group $G$ and assume $G$ acts algebraically and transitively on an algebraic manifold $X$ such that there exists $x_0\in X$ with $H_1x_0=x_0$. Let $\Omega=G^0_1\cap \Gamma_1$. Then there exists a connected Lie subgroup $L$ of $G^0_1$ such that $\Omega$ is a cocompact lattice in $L$, the commutator subgroup $[G_1,G_1]$ is in $L$, and $Lx=G_1x$ for all $x\in X$.}
\end{lemma}
\begin{proof} As before, let $G_2= [G_1, G_1]$ and $\Gamma_2= \Gamma_1 \cap G_2$. Then $\Gamma_2$  is Zariski dense in $G_2$. Since $G_2$ is unipotent, $\Gamma_2$ is a cocompact lattice. Let $\Lambda=\Omega/\Gamma_2$ and $T=G_1/G_2$. Then $\Lambda$ is a finitely generated subgroup of the abelian group $T$. Using one-parameter subgroups, we can easily construct a connected Lie subgroup $A$ of $T$ such that $\Lambda$ is a cocompact subgroup of $A$. Now, we let $L$ be the preimage of $A$ under the epimorphism of $G_1$ onto $T$. Clearly, $G_2\leq L$ and $\Omega$ is a cocompact lattice in $L$.

For a given $x\in X$, there exists a maximal reductive subgroup of $G_1$ that fixes it, which implies $G_1x=U_1x$. We claim that $Lx=U_1x$. To see this, it suffices to show that the standard projection from $G_1$ onto $U_1$ (mapping $u_1h_1$ to $u_1$ for $u_1\in U_1$ and $h_1\in H_1$) remains onto when restricted to $L$. By taking quotients with $G_2$, we can reduce  to the abelian case where $T$ maps onto its unipotent radical $V$ via the standard projection $\pi:T\to V$. Since  $\Lambda$ is a finite index subgroup of $\Gamma_1/\Gamma_2$, the projection $\pi(\Lambda)$ is a finite index subgroup of $\pi(\Gamma_1/\Gamma_2)$.  Because $\Gamma_1/\Gamma_2$ is  Zariski dense in $T$, $\pi(\Gamma_1/\Gamma_2)$ is Zariski dense in $V$. This implies that  $\pi(\Lambda)$ and hence $\pi(A)$ are Zariski dense in $V$. Since $\pi(A)$ is a connected Lie subgroup of  $V$, we deduce that $\pi(A)=V$.
\end{proof}

The following lemma provides us with a key tool to use an induction approach for the study of algebraic crystallographic actions.

\begin{lemma}\label{lem: lem6} {\normalfont Suppose $\Gamma$ is a Zariski dense subgroup of an algebraic group $G$. Let $\Gamma_1$ be a solvable normal subgroup of $\Gamma$ and $G_1$ be the algebraic closure of $\Gamma_1$. Let $U_1$ be the unipotent radical and $H_1$ be a maximal reductive subgroup of $G_1$. Suppose $G$ acts algebraically on a contractible algebraic manifold $X$ such that there exists $x_0\in X$ with $Hx_0=x_0$ where $H$ is a maximal reductive subgroup of $G$ containing $H_1$. Suppose $\Gamma$ acts crystallographically on $X$. Then,
\begin{enumerate}[(a)]
\item $G_1x_0$ is $G_1$-equivariantly isomorphic, as an algebraic manifold, to  $U_1$ with $G_1$-action analogous to (\ref{eq: extension2}).
\smallskip

\item $\widetilde{X}=X/{G_1}$ is $G$-equivariantly  isomorphic, as an algebraic manifold, to a quotient of $U$ by a closed connected subgroup where $G$ acts as in (\ref{eq: extension2}).
\smallskip

\item The group $\widetilde{\Gamma}= \Gamma/\Gamma_1$ is a crystallographic group of motions of $\widetilde{X}$.
\smallskip
\end{enumerate}}
\end{lemma}

\begin{proof} For part (a), Lemma \ref{lem: lem7} shows that $G_1x_0$ is isomorphic to  $U_1$.  The fact that this isomorphism can be chosen to be $G_1$-equivariant will be evident from the proof of part (b).

To prove (b), we note that, by Lemma \ref{lem: lem5}, $X=Ux_0$. The action of $G$ on $\widetilde{X}$ is transitive. It is not difficult to see that the isotropy group $G_{\tilde{x}_0}$ at the point $\tilde{x}_0=G_1x_0$ of $\widetilde{X}$ is $VH$ where $V= U\cap G_{\tilde{x}_0}$. We define a left action of $G$ on the homogeneous space of left cosets $U/V$ by
\begin{equation} wh\cdot [u] = [whuh^{-1}], \hspace{2mm}\forall u,w\in U, \forall h\in H. \label{eq: extension2}
\end{equation}
Then, we have a $G$-equivariant isomorphism of algebraic manifolds $$\widetilde{X}\cong UH/VH\cong U/V.$$ Here, the second isomorphism is defined by $[uh]\mapsto [u]$,  for all $u\in U$, $h\in H$. This proves part (b).

For (c), let $K$ be a compact subset of $\widetilde{X}$. Let $K'$ be a compact subset of $X$ mapped onto $K$ by the projection of $X$ onto $\widetilde{X}$.  Let us consider the sets
\begin{align*}
\mathscr{B}=&\{\widetilde{\gamma}\in \widetilde{\Gamma}\;|\; \widetilde{\gamma}K\cap K\ne \emptyset \},\\
\mathscr{C}=&\{{\gamma}\in \Gamma\;|\; {\gamma}K'\cap G_1K'\ne \emptyset \}.\\
\intertext{Let $\phi:\Gamma\rightarrow \widetilde{\Gamma}$ be the quotient map. It is not difficult to see that $\phi(\mathscr{C})=\mathscr{B}$.
\newline
\indent Now, from Lemma \ref{lem: lem5.5} it follows that $G_1K'=\Omega CK'$ for some compact subset $C$ of $G^0_1$ and $\Omega=G^0_1\cap \Gamma_1$.  We define}
\mathscr{D}=&\{{\gamma}\in \Gamma\;|\;{\gamma}CK'\cap CK'\ne \emptyset \}.
\end{align*}

\noindent By our choice of the compact set $C$ it follows that $\mathscr{C}\subseteq \Gamma_1 \mathscr{D}$. Since $\Gamma$ acts properly discontinuously on $X$, $\mathscr{D}$ is a finite set. It follows that $\mathscr{B}\subseteq \phi(\Gamma_1\mathscr{D})=\phi(\mathscr{D})$ and  therefore, $\mathscr{B}$ is also finite. This proves that $\widetilde{\Gamma}$ acts properly discontinuously on $\widetilde{X}$.

To show that $\widetilde{\Gamma}$ acts cocompactly, we note that $\widetilde{X}/\widetilde{\Gamma}\cong \widetilde{X}/\Gamma$ where $\Gamma$ acts through the quotient $\Gamma/\Gamma_1$. As $\widetilde{X}/\Gamma$ is an image of the compact set $X/\Gamma$, it is clearly compact.
\end{proof}


When considering algebraic crystallographic actions, by means of the previous lemma, we will be able to assume that the crystallographic group  $\Gamma$ intersects trivially with the radical $R$ of $G$ and consequently that it embeds as a discrete subgroup in $G/R$. We will therefore shift our attention to actions of semisimple and more generally of
reductive algebraic groups. The last part of this section is devoted to the necessary results in this direction.

\begin{lemma}\label{lem: lem8} {\normalfont Suppose $G$ is an algebraic group. Let $R$ be the radical of $G$ and  let $S$ be the semisimple group  $G^0/R$. Suppose $\Gamma$ is a discrete subgroup of $G^0$. Set $\widetilde{\Gamma}=\Gamma/{R\cap \Gamma}$. Then $\widetilde{\Gamma}$ acts
properly discontinuously on the symmetric space $X_S$ of $S$.}
\end{lemma}

\begin{proof} Since $\Gamma$ is a discrete subgroup of $G^0$, by Corollary 5.4 in \cite{Abels}, it follows that $\widetilde{\Gamma}$ is a discrete subgroup of $S$.  Let $C$ be an arbitrary compact subset of $X_S$. Let $C'$ be a compact subset of $S$ such that $\pi(C')=C$ where $\pi :S\rightarrow X_S$ is the quotient map. Let $K$ be a maximal compact subgroup of $S$ and consider the sets
\begin{align*}
\mathscr{E}=&\{\widetilde{\gamma}\in \widetilde{\Gamma}\;|\; \widetilde{\gamma}C\cap C\ne \emptyset\},\\
\mathscr{F}=&\{\widetilde{\gamma}\in \widetilde{\Gamma}\;|\; \widetilde{\gamma}C'K\cap C'K\ne \emptyset\}.
\end{align*}
Since $\widetilde{\Gamma}$ is discrete, it acts properly discontinuously on $S$. Therefore, $\mathscr{E}= \mathscr{F}$ is finite.
\end{proof}

\begin{lemma}\label{lem: lem9} {\normalfont Let $\lieu$ be a nilpotent real Lie algebra and let $\liev$ be a subalgebra of $\lieu$ that does
not contain any non-zero ideal of $\lieu$. Assume that $\varphi\in \Aut(\lieu)$ is such that
$\varphi(\liev)=\liev$ and denote by $\bar{\varphi}$ the induced linear map on $\lieu/\liev$.
If the restriction $\varphi_{|\liev}$ of $\varphi$ to $\liev$ is not unipotent, then
$\bar{\varphi}$ is also not unipotent.}
\end{lemma}

\begin{proof}Let us consider the complex Lie algebra $\lieu_\C= \lieu\otimes \C$ and
its subalgebra $\liev_\C=\liev \otimes \C$. It is obvious that
$\liev_\C$ also does not contain any non-zero ideal of $\lieu_\C$.
It follows that it is enough to prove the lemma in case of complex Lie algebras
instead of real Lie algebras. Therefore, in what follows, we will drop the subscripts $\C$ from
our notation and assume that $\lieu$ and $\liev$ are Lie algebras over $\C$.

Let $A=\{ \psi\in \Aut(\lieu) \;|\; \psi(\liev)=\liev\}$, then $A$ is an algebraic subgroup
of $\Aut(\lieu)$. It follows that for every $\psi\in A$, its semisimple part $\psi_s$ and
its unipotent part $\psi_u$ are in $A$. Now, for the given $\varphi$ we can also
consider $\varphi_s$ and it is obvious that $\varphi{|\liev}$ is not unipotent if and only if the restriction
$\varphi_{s|\liev}$  of $\varphi_s$ to $\liev$
is not unipotent.
Moreover, the induced map of $\varphi$ on $\lieu/\liev$ is unipotent
if and only if the induced map of $\varphi_s$ on $\lieu/\liev$ is trivial.
So, it suffices to show that there is no semisimple automorphism $\varphi\in \Aut(\lieu)$, with
$\varphi_{|\liev}$ non-trivial and $\bar{\varphi}$ trivial.

We will prove this by contradiction and so we assume that such a semisimple $\varphi$ exists.
Let $c$ denote the nilpotency class of $\lieu$ and let $\gamma_1(\lieu)=\lieu$ and $\gamma_{i+1}(\lieu)=[\lieu,\gamma_i(\lieu)]$ denote
the terms of the lower central series of $\lieu$. Then, $\gamma_{c}(\lieu) \neq 0$ and $\gamma_{c+1}(\lieu)=0$.

We define the following subalgebras of $\lieu$:

\[ \forall i\in \{1,2,\ldots,c\}:\; \liev_i= \gamma_i(\lieu) \cap \liev.\]
\[ \forall i\in \{1,2,\ldots,c\}:\; \lieu_i=\liev+\gamma_i(\lieu) .\]
Then, we have $\liev_1=\liev$ and $\lieu_1=\lieu$. In this way, we obtain a filtration of $\lieu$:
\[ 0\subseteq \liev_c \subseteq \cdots \subseteq \liev_2 \subseteq \liev_1 \subseteq \lieu_c \subseteq \cdots \subseteq \lieu_2 \subseteq \lieu_1.\]
Moreover, each term of this filtration is preserved by $\varphi$. Let $\liev_{c+1}=0$ and $\lieu_{c+1}=\liev=\liev_1$ and denote
by $l_i=\dim \liev_i/\liev_{i+1}$ and $k_i=\dim \lieu_i/\lieu_{i+1}$ for $1\leq i \leq c$.
As $\varphi$ is semisimple, we can choose a basis of the vector space $\lieu$ which consists of vectors
\[ V_{i,j}\in \liev \mbox{ ($1\leq i \leq c$, $1\leq j \leq l_i$) and }
 U_{p,q}\in \lieu \mbox{ ($1\leq p \leq c$, $1\leq q \leq k_i$)}\]
where each of these vectors is an eigenvector for $\varphi$ and
\begin{align*}
&\liev_r =  {\rm span}\{V_{i,j} \;|\; r \leq i \leq c,\; 1\leq j \leq l_i\} & \mbox{ for all } r\in \{1,2,\ldots,c\},\\
&\lieu_r = {\rm span}\{U_{p,q} \;|\; r \leq p \leq c,\; 1\leq q \leq k_i\}+ \liev & \mbox{ for all } r\in \{1,2,\ldots,c\}.
\end{align*}
We use $\lambda_{i,j}\in \C$ to denote the eigenvalue corresponding to the eigenvector $V_{i,j}$. By the condition on $\varphi$, we know that there
is at least one pair $(i,j)$, with $\lambda_{i,j}\neq 1$. Also, $\varphi(U_{p.q}) = U_{p.q}$, since $\varphi$ induces the identity on $\lieu/\liev$.
Let $$i_0={\rm max}\{i\in \{1,2,\ldots, c\}\;|\; \exists j\in \{1,2,\ldots, l_c\}:\;\lambda_{i,j}\neq 1\}$$ and fix a
$j_0\in \{1,2,\ldots l_c\}$ with $\lambda_{i_0,j_0}\neq 1$. As a shorthand, we will use $\lambda_0=\lambda_{i_0,j_0}$ and
$V_0=V_{i_0,j_0}$. Note that the choice of $i_0$ implies that $\varphi$ is the identity on $\liev_{i_0+1}$.
Because of the fact that $\liev$ is not an ideal of $\lieu$,
we know that there exists a vector $U\in \lieu$ with $[U,V_0]\neq 0$. So, there must also exist a basis vector $U$ ($V_{i,j}$ or $U_{p,q}$)
with $[U,V_0]\neq 0$. We distinguish two possibilities, each of which will lead to a contradiction.

\medskip

\noindent\underline{{\bf Case 1:} There exists a vector $U_{p,q}$ with $[U_{p,q},V_0]\neq 0$.}\\
Let $X=[U_{p,q},V_0]$, then $\varphi(X)=[\varphi(U_{p,q}),\varphi(V_0)]= [U_{p,q},\lambda_0 V_0]=\lambda_0 X$, hence
$X$ is a (non-zero) eigenvector for the eigenvalue $\lambda_0$. If $X\not\in\liev$, this contradicts the fact that
$\varphi$ induces the identity on $\lieu/\liev$. Hence, we must have that $X\in \liev$, but then $X\in \liev_{i_0+1}$, contradicting
the fact that, by our choice of $i_0$, $\varphi$ induces the identity on $\liev_{i_0+1}$.

\medskip

\noindent\underline{{\bf Case 2:} There is no vector $U_{p,q}$ with $[U_{p,q},V_0]\neq 0$, but $[V_{i,j},V_0]\neq 0$ for some pair $(i,j)$.}\\
Let $i_1={\rm max}\{i\in \{1,2,\ldots, c\}\;|\;\exists j \in \{1,2,\ldots, l_c\}:\; [V_{i,j}, V_0]\neq 0\}$ and choose a $j_1\in \{1,2,\ldots,l_{i_1}\}$
with $[V_{i_1,j_1},V_0]\neq 0$. We use the abbreviation $V_1=V_{i_1,j_1}$ and $\lambda_1= \lambda_{i_1,j_1}$. By the choice of
$i_1$, we know that $[\liev_{i_1+1}, V_0]=0$. We observe that $\lambda_1 \lambda_0=1$, otherwise, $[V_1,V_0]\in \liev_{i_0+1}$ would be
a (non-zero) eigenvector for $\varphi$ for the eigenvalue $\lambda_0 \lambda_1\neq 1$, contradicting the choice of $i_0$.
So, we know that $\lambda_1=1/\lambda_0\neq 1$. As $[V_1,V_0]\neq 0$, there exists a basis vector $X$ of $\lieu$ with
$[ X,[V_1,V_0]]\neq 0$. By the Jacobi identity, we have that
\begin{equation}\label{jacobi}
 [X,[V_1,V_0]]= - [V_1,[V_0,X]]-[V_0,[X,V_1]]
\end{equation}
First suppose that $X=U_{p,q}$ for some values of $p$ and $q$. Since $[V_0,U_{p,q}]=0$,
this implies $[U_{p,q},V_1]\neq 0$. As $[U_{p,q},V_1]$ is an eigenvector of eigenvalue $\lambda_1\neq 1$, we must have that
$[U_{p,q},V_1]\in \liev$. Moreover, $[U_{p,q},V_1]\in \gamma_{i_1+1}(\lieu)$ and hence $[U_{p,q},V_1]\in \liev_{i_1+1}$. But, by the choice
of $i_1$, this implies that $[V_0,[U_{p,q},V_1]]=0$. This shows that both terms on the righthand side of (\ref{jacobi}) are 0, implying
that also $ [X,[V_1,V_0]]=0$, a contradiction.

Therefore, we may assume that $X=V_{i_2,j_2}$ for some pair $(i_2,j_2)$. We use the notation $V_2=V_{i_2,j_2}$ and $\lambda_2=\lambda_{i_2,j_2}$.
Note that $\varphi[V_2,[V_1,V_0]]= \lambda_2 [V_2, [V_1,V_0]]$. Then $ [V_2, [V_1,V_0]]$ is a nonzero eigenvector associated to the eigenvalue
$\lambda_2$. As $ [V_2, [V_1,V_0]]\in \liev_{i_0+1}$, this implies that $\lambda_2=1$.

Now, we look again at the right-hand side of (\ref{jacobi}), for $X=V_2$.
As $[V_2,V_1]\in \liev_{i_1+1}$, we immediately have that $[V_0,[V_2,V_1]]=0$, by the choice of $i_1$.
As $\varphi([V_0,V_2])=\lambda_0 \lambda_2 [V_0,V_2]=\lambda_0 [V_0,V_2] $ and $[V_0,V_2]\in \liev_{i_0+1}$, the choice
of $i_0$ forces $[V_0,V_2]=0$.
Again, both terms of the right-hand side of (\ref{jacobi}) are 0, leading to the fact that $[V_2,[V_1,V_0]]=0$, a contradiction.
\end{proof}

\begin{lemma}\label{lem: lem10} {\normalfont Let $U$ be a 1-connected nilpotent Lie group and
let $H\leq \Aut(U)$ be a reductive subgroup. Assume that $V\leq U$ is a closed and
connected subgroup of $U$ which is invariant under the action of $H$. Suppose that
the induced action of $H$ on $U/V$ is faithful. Then the corresponding action of $H$ on the tangent space $T_V(U/V)$ of $U/V$ at $eV$
is faithful.}
\end{lemma}

\begin{remark}{\normalfont First of all, we remark that the requirement that $H$ is reductive is necessary.
Let $U$ be the three-dimensional Heisenberg group
\[ \left\{ [x,y,z]:=\left( \begin{array}{ccc} 1 & x & z \\
0 & 1  & y \\ 0 & 0 & 1 \end{array} \right) \;|\; x,y,z \in \R \right\}\]
and let $V$ be the subgroup $\left\{ [x,0,0]\;|\; x \in \R \right\}$.
It follows that we can identify $U/V$ with $\R^2$ by the map
\[ \psi : U/V \rightarrow \R^2:  [x,y,z]V \mapsto (y,z).\]
Now, for all $t\in \R$, $$\phi_t:U\to U: [x,y,z] \mapsto
[x+t y , y,  z+ t y^2/2]$$
is an automorphism of $U$, with $\phi_t(V)=V.$ Each $\phi_t$ induces on $U/V$ a map
\[ \bar{\phi}_t: \R^2 \rightarrow \R^2 : (y,z) \mapsto (y, z + t y^2/2).\]
Note that the differential of $\bar{\phi}_t$ at $(0,0)$ is the identity map and that
$\{\phi_t \; | \; t\in \R\}$ is a unipotent subgroup of $\Aut(U)$
acting faithfully on $U/V=\R^2$.}
\end{remark}

\begin{proof}[Proof of Lemma \ref{lem: lem10}]It is not difficult to see that when $V$ is a normal subgroup of $U$, then the
proposition is trivial to prove. In fact, even more is true. Let us explain
that we can discard the maximal normal subgroup $N$ of $U$ which is contained in $V$.
It is easy to see that $N$ is a connected subgroup of $U$ and that for every
automorphism $\phi\in \Aut(U)$, with $\phi(V)=V$, we also have that $\phi(N)=N$.

There is a natural identification of the spaces $U/V$ and $(U/N)/(V/N)$,
under which the action of $H$ on $U/V$ can also be seen as an action of $H$ on $ (U/N)/(V/N)$.
It follows that by replacing $V$ by $V/N$ and $U$ by $U/N$, we can assume that
$V$ does not contain any non-trivial normal subgroup of $U$.

Let us now denote by $\lieu$ the Lie algebra of $U$ and by $\liev$ the Lie algebra of $V$, so
$\liev$ is a subalgebra of $\lieu$. Any automorphism $\phi\in \Aut(U)$, induces an automorphism
$d\phi\in \Aut(\lieu)$ and vice versa. Of course $\phi(V)=V$ if and only if $d\phi(\liev)=\liev$.
The fact that $V$ is a subgroup of $U$ which does not contain a non-trivial normal subgroup of $U$ is equivalent to the
fact that $\liev$ does not contain any non-zero ideal of $\lieu$. This, in turn, is equivalent to the fact that
$\liev\cap Z(\lieu) = 0$.


To be able to describe the action of $H$ on $U/V$, we will fix a diffeomorphism
of $U/V$ with a finite dimensional real vector space, using a weak Mal'cev basis of
the Lie algebra $\lieu$ of $U$. We suppose that
$\lieu$ is $c$-step nilpotent and define the subalgebras
$\lieu_i=\gamma_i(\lieu)+ \liev$ forming a descending sequence
\[ \lieu=\lieu_1\supseteq \lieu_2\supseteq \lieu_3 \supseteq \cdots \supseteq \lieu_c \supseteq \liev=\lieu_{c+1}\]
of subalgebras, which are invariant under the
action of $H$.
Moreover, for any $i$, $\lieu_{i+1}$ is an ideal in $\lieu_{i}$,
this will turn out to be crucial in what follows.

As in the proof of the previous lemma, we use $k_i$ to denote the dimension of $\lieu_i/\lieu_{i+1}$ and choose for each $i$ elements
$X_{i,1}, X_{i,2}, \ldots, X_{i,k_i}$ in $\lieu_i$ is such a way that the natural
projections of these elements form a basis of $\lieu_i/\lieu_{i+1}$. Then each element
$u\in U$ can be uniquely written under the form

\begin{equation} \exp(t_{1,1} X_{1,1})\exp(t_{1,2} X_{1,2})\cdots \exp(t_{1,k_1}X_{1,k_1})\exp(t_{2,1} X_{2,1}) \cdots \exp(t_{c,k_c} X_{c,k_c})\label{writingu}
\exp( Y)
\end{equation}
where $t_{i,j}\in \R$ and $Y\in \liev$.
The map $$\psi: U/V \rightarrow \R^{k_1+k_2+\cdots+k_c}:\,u\mapsto (t_{1,1}, t_{1,2}, \ldots, t_{1,k_1}, t_{2,1}, \ldots, t_{c,k_c}),$$
where the $t_{i,j}$ are as in (\ref{writingu}), is a diffeomorphism sending the point $eV$ to the origin.

Now, consider an automorphism $\phi\in H$. Assume that
there exists an $i\in \{1,2,\ldots c\}$ such that $\phi$ (or rather $d\phi$) does not induce the identity
on $\lieu_i/\lieu_{i+1}$. This means that there are real numbers
$a_{p,q}$, $1\leq p,q \leq k_i$ so that
$d\phi(X_{i,q}) =  \sum_{p=1}^{k_i} a_{p,q} X_{i,p}$ and the $k_i\times k_i$--matrix $A=(a_{p,q})$ is not the identity matrix.
Using the fact that $\lieu_{i+1}$ is an ideal in $\lieu_i$, we find that
\[ \varphi\left(  \exp(t_{i,1} X_{i,1}) \exp(t_{i,2} X_{i,2}) \cdots \exp(t_{i,k_i} X_{i,k_i})\right) =\]
\[ \exp(\sum_q a_{1,q}t_{i,q} X_{i,1})\exp(\sum_q a_{2,q} t_{i,q} X_{i,2})\cdots
\exp (\sum_q a_{k_i,q} t_{i,q} X_{i,k_i}) \cdot \alpha\]
for some $\alpha \in \exp(\lieu_{i+1})$. This means that for the action of $\phi$ on the space
$U/V=\R^{k_1+\cdots+k_c}$, we have that
\[ \phi\cdot (0,0,\ldots,0,t_{i,1},t_{i,2}, \cdots, t_{i,k_i}, 0,0,\ldots,0)=\]
\[ (0,0,\ldots,0,\sum_q a_{1,q}t_{i,q} ,\sum_q a_{2,q} t_{i,q}, \ldots, \sum_q a_{k_i,q} t_{i,q} , \ast,\ast, \ldots,\ast ).\]
It follows that the differential of this map at the origin is non-trivial.
Thus, if we can show that for any $\phi\in H$ there exists an $i\in \{1,2,\ldots,c\}$ such that
$\phi$ does not induce the identity on $\lieu_i/\lieu_{i+1}$ we are done.
Therefore, we consider the morphism
\begin{equation} \psi: H \rightarrow \GL(\lieu_1/\lieu_2)\times \GL(\lieu_2/\lieu_3) \times \cdots \times \GL(\lieu_c/\lieu_{c+1})\label{mainrep}
\end{equation}
mapping any $\phi$ in $H$ to the $c$-tuple of maps it induces on the consecutive quotients.  Suppose, by a way of contradiction, that $K={\rm Ker}(\psi)$ is non-trivial. Then, as $H$ is a reductive
group, so is $K$. Since $K$ is non-trivial, there must exist an element $\phi\in K$ which is
not unipotent. It follows that also $d\phi$ is not a unipotent element. However, the fact that $\phi\in K$ implies that
$d\phi$ acts trivially on each of the factors $\lieu_i/\lieu_{i+1}$, implying that $d\phi$ induces a unipotent map on $\lieu/\liev$.
 On the other hand, since $d\phi$ is not a unipotent element, this also implies that $d\varphi_{|\liev}$ is not unipotent.
But, now Lemma \ref{lem: lem9} tells us that $\overline{d\phi}$ is not unipotent on $\lieu/\liev$, which is a contradiction. It follows that $K$ is trivial and this finishes the proof.
\end{proof}

In the next lemma, we derive the ``eigenvalue one criteria'' on the induced representation of the reductive subgroup of the given algebraic group
which will be essential in our subsequent arguments.

\begin{lemma}\label{lem: lemp11} {\normalfont Let $\Gamma$ be a Zariski dense subgroup of an algebraic group $G$. Suppose $G$ acts algebraically on an algebraic manifold $X$ such that there exists $x_0\in X$ fixed by the action of a maximal reductive subgroup $H$ of $G$. Let $\rho:H\rightarrow \mbox{GL}(T_{{x}_0}{X})$ be the representation corresponding to the action of $H$ on the tangent space $T_{{x}_0}{X}$ at $x_0$. If $\Gamma$ acts freely  on $X$, then $\rho(h)$ has an eigenvalue equal to one for every $h\in H$.}
\end{lemma}

\begin{proof} The fact that $H$ acts algebraically on $X$ implies that $\rho$ is a morphism of algebraic groups. Let $\psi:G\rightarrow H$ denote the projection of $G$ onto $H$.  Next, we will show that for each $\gamma\in \Gamma$, $\rho\circ \psi(\gamma)$ has an eigenvalue equal to one. Since $\psi(\Gamma)$ is Zariski dense in $H$, this will give us the desired result.

For this purpose, we let $\gamma=\gamma_s\gamma_u$ be the Jordan decomposition of a nontrivial element $\gamma\in \Gamma$, where $\gamma_s$ is semisimple and $\gamma_u$ is unipotent. There exists $u\in U$ such that $u\gamma_s u^{-1}=h_0\in H$.  We let $\gamma_u^t$ be the one-parameter subgroup of $\gamma_u$ and define a curve in $G$ by $\gamma(t)=\gamma^t_u\gamma_s$, $t \in \R$. Since $\Gamma$ acts freely, the curve $\gamma(t)$ does not fix the point ${x}_1=u^{-1}{x}_0$ of ${X}$. Analogously to the representation $\rho$, we can construct a representation $\beta : u^{-1}Hu\rightarrow\mbox{GL}(T_{{x}_1}{X})$. By identifying the tangent spaces with an isomorphism, it follows that $\rho$ and $\beta$ are conjugate by $du:T_{{x}_1}{X}\rightarrow T_{{x}_0}{X}$. Since $\gamma_s$ fixes ${x}_1$, it also fixes the curve $\gamma(t){x}_1$ pointwise. It follows that $\beta(\gamma_s)$ has an eigenvalue equal to one.  Since $h_0=\psi(\gamma_s)$, we also have that $\rho\circ\psi(\gamma_s)$ has  an eigenvalue equal to one. But $\rho\circ\psi(\gamma)=\rho\circ\psi(\gamma_s)\rho\circ\psi(\gamma_u)$ and $\rho\circ\psi$ preserves Jordan decomposition (see 4.4(c) of \cite{Borel}). Hence, $\rho\circ\psi(\gamma)$ has an eigenvalue equal to one.
\end{proof}

We end this section by a lemma on special real representations of semisimple groups. For its proof, we refer to Lemma 5 of \cite{Soifer}.

\begin{lemma}\label{lem: lemp3} {\normalfont Suppose $S$ is a nontrivial connected semisimple Lie group such that the real rank of any simple subgroup of $S$ is at most one. Denote by $X_S$ the symmetric space of $S$. Let  $\rho : S\rightarrow \mbox{GL}_n(\mathbb R)$ be a faithful representation of $S$ such that for every $s\in S$, $\rho(s)$ has an eigenvalue equal to one.  Then $n>\mbox{dim}(X_S)$.}
\end{lemma}

\section{Main Results}

We begin with a generalization of \cite[Theorem A]{Soifer} and \cite[Theorem A]{Tomanov}.

\begin{theorem} \label{thm: main} {\normalfont Let $\Gamma$ be a  subgroup of an algebraic group $G$. Let $G=UH$ where $U$ is the unipotent radical and $H$ is a maximal reductive subgroup. Suppose $G$ acts algebraically on a contractible algebraic manifold $X$ such that there exists $x_0\in X$ with $Hx_0=x_0$. Suppose $\Gamma$ acts crystallographically on $X$.
\begin{enumerate}[(a)]
\item If the real rank of any simple subgroup of $G$ does not exceed one, then $\Gamma$ is virtually polycyclic.
\smallskip
\item If $\Gamma$ is a virtually polycyclic, Zariski dense subgroup of $G$, then $X$ is $G$-equivariantly isomorphic to $U$ with the $G$-action  given by
$$wh\cdot u = whuh^{-1}, \hspace{2mm}\forall u,w\in U, \forall h\in H.$$
\end{enumerate}}
\end{theorem}
\begin{proof}We first prove (a).  Since ${\Gamma}$ is a finitely generated linear group, in view of Selberg's lemma, we can assume that it is torsion-free. Also, without loss of generality, we can assume that  $\Gamma$ is Zariski dense in $G$. Let $R$ be the solvable radical of $G$ and let $S$ be a maximal connected semisimple Lie subgroup of $H$.  We set $\Gamma_1= \Gamma\cap R$ and denote by $G_1$ the algebraic closure of $\Gamma_1$ in $G$. Also, let $\widetilde{G}=G/G_1$, $\widetilde{\Gamma}=\Gamma/\Gamma_1$,  and $\widetilde{X}=X/G_1$.

By lemmas \ref{lem: lem7} and \ref{lem: lem6}, $\Gamma_1$ acts crystallographically on $G_1x_0$ and  $\widetilde{\Gamma}$ acts crystallographically on $\widetilde{X}$. Since $\widetilde{\Gamma}$ is finitely generated linear group, we can assume that it is torsion-free and hence acts freely on $\widetilde{X}$. Let $L$ be the kernel of the action of $G$ on $\widetilde{X}$. Then $L$ is a normal algebraic subgroup containing $G_1$. We define $G'=G/L$ and its subgroup $\Gamma'=\Gamma L/L$.  Since $\widetilde{\Gamma}$ acts freely on $\widetilde{X}$, $\Gamma\cap L=\Gamma\cap G_1$ and thus $\Gamma'\cong \widetilde{\Gamma}$. Therefore, it is enough to show that $\Gamma'$ is virtually polycyclic.

Observe that $\Gamma'$, $G'$ and $\widetilde{X}$ satisfy the hypothesis of the theorem in place of $\Gamma$, $G$ and $X$, respectively. Therefore, proceeding by induction on the dimension of $G$, we can assume that
$L$ is trivial.

Next, let us consider the tangent space $T_{{x}_0}{X}$ of ${X}$ at ${x}_0$ and the action of $S$ on $T_{{x}_0}{X}$.  We denote by $\rho:S\rightarrow \mbox{GL}(T_{{x}_0}{X})$ the representation corresponding to this action. Since $S$ acts faithfully on ${X}$,  Lemma \ref{lem: lem6}(b) together with Lemma \ref{lem: lem10} show that $\rho$ is a faithful representation.  Since $\Gamma$ intersects trivially with the solvable radical of $G$, by Lemma \ref{lem: lem8}, a finite index subgroup of $\Gamma$ acts properly discontinuously on the symmetric space  $X_{G^0/R}$.  This implies  $\mbox{cd}(\Gamma)\leq \mbox{dim}(X_{G^0/R})\leq \mbox{dim}(X_{S})$. Also, since $\Gamma$ acts crystallographically on ${X}$, $\mbox{cd}(\Gamma)= \mbox{dim}({X})$. Hence, we have
$\mbox{dim}(\rho)=\mbox{dim}(T_{{x}_0}{X})=\mbox{dim}({X})=\mbox{cd}(\Gamma)\leq \mbox{dim}(X_{S})$.

On the other hand, by Lemma \ref{lem: lemp11}, it follows that for every $s\in S$, $\rho(s)$ has an eigenvalue equal to one. But, according to Lemma \ref{lem: lemp3}, this is only possible when  $S$ is trivial. We deduce that $\Gamma$ is virtually polycyclic.

If $\Gamma$ is a virtually polycyclic, Zariski dense subgroup of $G$, then the identity component of $G$ is solvable.  Assertion (b) now follows directly from Lemma \ref{lem: lem6}.
\end{proof}

\begin{remark}\label{rmk: aftermain}{\normalfont Assuming the hypothesis of part (b) of the theorem, let us consider the resulting action of the group $G$. Note that the conjugation action gives rise to a representation $\phi:H\rightarrow \mbox{Aut}(U)$. Let $H'$ be its image and denote by $G'$ the semidirect product $U\rtimes H'$. Then there is an epimorphism of algebraic groups $\psi:G\rightarrow G'$ given by $uh\mapsto u\phi(h)$ for each $u\in U$ and $h\in H$.  Now, the action of $G$ on $X$ factors through $G'$ and has kernel equal to $\mbox{ker}(\phi)$. Because $\Gamma$ is a crystallographic subgroup of $G$ and $X\cong U$, its image under $\psi$ is a NIL-affine crystallographic subgroup of $G'$ and $\Gamma \cap \mbox{ker}(\phi)$ is finite. This shows that the action of $\Gamma$ on $X$ is NIL-affine crystallographic.}
\end{remark}

As we mentioned earlier, every virtually polycyclic group admits a NIL-affine crystallographic action. This was proven independently by Dekimpe in \cite{dek} and Baues  in \cite{baues}. Next, we present a new proof of this result which  is a constructive application of Theorem \ref{thm: main}(b).

\begin{theorem} \label{thm: action} {\normalfont Let $\Gamma$ be a  virtually polycyclic group. Then it admits a NIL-affine crystallographic action.}
\end{theorem}
\begin{proof}It is well-known that any virtually polycyclic group is linear. Therefore, we can assume $\Gamma$ is a Zariski dense subgroup of an algebraic group $G$. Since $\Gamma$ is virtually solvable, $G^0$ is a solvable normal subgroup of $G$. Let $G_1=[G^0,G^0]$ and $\Gamma_1=\Gamma\cap G_1$. Then $G_1$ is a normal unipotent subgroup of $G$ and $\Gamma_1$ is Zariski dense in $G_1$ and hence it is a cocompact lattice. Let $\Lambda=\Gamma/\Gamma_1$. Then $\Lambda$ is a virtually abelian group. We denote by $\mbox{E}(n)$ the image of the standard embedding of the group of isometries of $\R^n$ into $\mbox{GL}_{n+1}(\R)$. By a theorem of Zassenhaus \cite{Zass}, we know that $\Lambda$ admits an Euclidean crystallographic action on some $\R^n$, i.e. there exists a representation $\rho:\Lambda\rightarrow \mbox{E}(n)$ with finite kernel such that $\rho(\Lambda)$ is crystallographic. Let $\overline{G}= G\times \mbox{E}(n)$ and note that we can  embed $\Gamma$ into $\overline{G}$ by the homomorphism $\gamma\mapsto (\gamma, \rho([\gamma]))$, $\forall\gamma \in \Gamma$. We identify $\Gamma$ with its image in $\overline{G}$.

Let $U$ be the unipotent radical of $G$. Next, we construct a transitive action of $\overline{G}$ on $X=U\times \R^n$ by
$$g\cdot (u, y)= (whuh^{-1}, x+\alpha y) \hspace{5mm} \forall g=(wh,(x, \alpha))\in \overline{G}, \forall(u, y)\in X.$$
This is clearly an algebraic action and it is not difficult to see that a maximal reductive subgroup of $\overline{G}$ fixes a point.
We claim that the restriction of the action to $\Gamma$ is properly discontinuous.

Suppose this is not the case. Then there exists  a compact subset $K=K_1\times K_2$ in $X$, where $K_1\subseteq U$ and $K_2\subseteq \R^n$ such that the intersection of the sets
\begin{align*}
\mathscr{G}=&\{\gamma\in \Gamma \;|\; \gamma K_1\cap K_1\ne \emptyset\},\\
\mathscr{H}=&\{\gamma\in \Gamma \;|\; [\gamma] K_2\cap K_2\ne \emptyset\}\\
\intertext{is infinite. Now, since $\Lambda$ acts properly discontinuously on $\R^n$, this implies that there is an
element $\gamma_0\in \Gamma$, for which}
\mathscr{I}=&\{\gamma \in \mathscr{G}\cap \mathscr{H} \;|\; [\gamma]=[\gamma_0] \}\\
\intertext{is infinite. Hence, also the set}
&\{\gamma_0^{-1} \gamma \in \Gamma_1  \;|\; \gamma_0^{-1}\gamma K_1\cap \gamma_0^{-1}K_1\ne \emptyset\},
\end{align*}
containing $\gamma_0^{-1}\mathscr{I}$,
is infinite. This is a contradiction to the fact that $\Gamma_1$ is a discrete subgroup of $U$.

Next, let  $\overline{\Gamma}$ be the algebraic closure of $\Gamma$ in $\overline{G}$ and denote by $H'$ a maximal reductive subgroup of $\overline{\Gamma}$. Then $H'$ is a reductive subgroup of $\overline{G}$ and hence it is contained in a maximal reductive subgroup $H$ of $\overline{G}$. Let $x_0$ be a point fixed by $H$ and consider the orbit $Y$ at $x_0$ of the action of $\overline{\Gamma}$ on $X$. It is clearly a contractible algebraic submanifold of $X$. Since $\overline{\Gamma}^0$ is solvable and $Y=\overline{\Gamma}^0x_0$, by Lemma \ref{lem: lem6.5}, it follows that $\Gamma$ acts crystallographically on $Y$. Therefore, Theorem \ref{thm: main} and Remark \ref{rmk: aftermain} show that $\Gamma$ admits a NIL-affine crystallographic action on the unipotent radical of $\overline{\Gamma}$ which is isomorphic to $Y$.
\end{proof}
Now, we consider algebraic crystallographic actions in the case where the ambient space $X$ is the 1-connected nilpotent Lie group $N$ and the algebraic group is the group affine motions of  $N$. Recently, Abels, Margulis, and Soifer have shown that Auslander's conjecture holds up to dimension six (see \cite{AMS}). Building on this result, we prove that the generalized Auslander conjecture for NIL-affine actions also holds up to dimension six. But first, we need a lemma.

\begin{lemma}\label{lem: prop} {\normalfont Let $N$ be a 1-connected nilpotent Lie group and suppose $H$ is a maximal reductive subgroup of affine transformations $\mbox{Aff}(N)$. Then there exists $x_0\in N$ such that $Hx_0=x_0$.}
\end{lemma}

\begin{proof}Let $U$ be the unipotent radical of $\mbox{Aff}(N)$. Consider the natural epimorphism of algebraic groups $\pi:\mbox{Aff}(N)\rightarrow \mbox{Aut}(N)$. Since $\mbox{ker}(\pi)=N$ is unipotent, $\pi_{|H}:H\rightarrow \mbox{Aut}(N)$ is injective. As any two maximal reductive subgroups are conjugate, we can assume $H\leq \mbox{Aut}(N)$. Hence, $He=e$ finishing the claim.
\end{proof}

The lemma together with Theorem \ref{thm: main} give us the following.

\begin{theorem} \label{thm: app1} {\normalfont Let $\Gamma$ be a NIL-affine crystallographic group of a 1-connected nilpotent Lie group $N$. If the real rank of any simple subgroup of the algebraic closure of $\Gamma$ does not exceed one, then $\Gamma$ is virtually polycyclic and  $N$ is isomorphic to the unipotent radical of the algebraic closure of $\Gamma$.}
\end{theorem}

\begin{theorem} \label{thm: app2} {\normalfont Suppose $\Gamma$ is a NIL-affine crystallographic group of a 1-connected nilpotent Lie group $N$ of dimension at most six. Then $\Gamma$ is virtually polycyclic and  $N$ is isomorphic to the unipotent radical of the algebraic closure of $\Gamma$.}
\end{theorem}

\begin{proof} In \cite{BDD}, this result has been proven when $\mbox{dim}(N)\leq 5$. In the abelian case,  $N=\mathbb R^6$, Auslander's conjecture has been settled (see \cite{AMS}). Also, if $N$ is 2-step nilpotent, then by Proposition 3 of \cite{BDD}, $\Gamma$ admits an affine crystallographic action on $\mathbb R^6$. So, by the abelian case, we can again conclude that $\Gamma$ is virtually polycyclic.

Now,  we observe that in dimension six, there are 33 non-isomorphic, non-abelian, nilpotent real Lie algebras  (see \cite{Magnin}, \cite{OV}). For 30 of them, the Lie algebra $\lien$ corresponding to $N$ is either 2-step nilpotent or has solvable derivation algebra $\mbox{Der}(\lien)$. When $\mbox{Der}(\lien)$ is solvable, since it is the Lie algebra of $\mbox{Aut}(\lien)$, it directly follows that $\mbox{Aff}(N)$ is virtually solvable. Therefore, in these 30 cases, we can deduce that $\Gamma$ is virtually polycyclic.

Using Theorem \ref{thm: app1}, we will show that in the remaining three cases, $\Gamma$ is also virtually polycyclic. (To describe the three cases, we use the notation of
\cite{OV}).

\medskip

\noindent\underline{{\bf Case 1:} $ \lieg_{6,12}$: $[X_1,X_2]=X_5$, $[X_1,X_5]=X_6$, $[X_3,X_4]=X_6$.}\\[2mm]
By writing the elements of the Lie algebra as column vectors using basis $\{X_6, \dots, X_1\}$, we find

$$\mbox{Aut}(\lieg_{6,12})=\displaystyle{\left\{\left( \begin{array}{llllll} \alpha^2_1\beta_2 & \alpha_1\varepsilon_2 & \lambda_4 & \lambda_3 & \lambda_2 & \lambda_1
\\ 0 & \alpha_1\beta_2 & \varepsilon_4 &  \varepsilon_3 & \varepsilon_2 &  \varepsilon_1\\ 0 & 0 & \partial_4 & \partial_3 & 0 &
\partial_1\\ 0 & 0 & \gamma_4  & \gamma_3 & 0 & \gamma_1\\ 0 & 0 & 0 & 0 & \beta_2 &  \beta_1
\\ 0 & 0 & 0 & 0 & 0 &  \alpha_1\end{array}\right)\right\}}<  \mbox{GL}_6(\mathbb R),$$ where $\partial_1=
\displaystyle\frac{\partial_4\varepsilon_3-\partial_3\varepsilon_4}{\alpha_1\beta_2},$ $\gamma_1=\displaystyle\frac{\gamma_4\varepsilon_3-\gamma_3\varepsilon_4}{\alpha_1\beta_2}$, $\alpha_1\beta_2\ne 0$,
 and $\partial_4\gamma_3-\partial_3\gamma_4=\alpha^2_1\beta_2$.


\noindent Let $\varphi_1: \mbox{Aut}(\lieg_{6,12})\rightarrow \mbox{GL}_6({\mathbb R})$ be the morphism defined by
$$\displaystyle{\left(\begin{array}{llllll} \alpha^2_1\beta_2 & \alpha_1\varepsilon_2 & \lambda_4 & \lambda_3 & \lambda_2 & \lambda_1
\\ 0 & \alpha_1\beta_2 & \varepsilon_4 &  \varepsilon_3 & \varepsilon_2 &  \varepsilon_1\\ 0 & 0 & \partial_4 & \partial_3 & 0 &
\partial_1\\ 0 & 0 & \gamma_4  & \gamma_3 & 0 & \gamma_1\\ 0 & 0 & 0 & 0 & \beta_2 &  \beta_1
\\ 0 & 0 & 0 & 0 & 0 &  \alpha_1\end{array}\right)\mapsto \left(\begin{array}{llllll} \alpha^2_1\beta_2 & 0 & 0 & 0 & 0 & 0
\\ 0 & \alpha_1\beta_2 & 0 &  0 & 0 &  0\\ 0 & 0 & \partial_4 & \partial_3 & 0 &
0\\ 0 & 0 & \gamma_4  & \gamma_3 & 0 & 0\\ 0 & 0 & 0 & 0 & \beta_2 & 0
\\ 0 & 0 & 0 & 0 & 0 &  \alpha_1\end{array}\right)}.$$

\noindent The kernel of $\varphi_1$ is unipotent. We note that $\mbox{Im}(\varphi_1)\cong \mbox{GL}_2(\mathbb R)\times \R^{*}\times \R^{*}$ and
hence has semisimple rank one. Therefore, by Theorem \ref{thm: app1}, $\Gamma$ is virtually polycyclic.

\medskip

\noindent\underline{{\bf Case 2:} $ \lieg_{5,4}\oplus \mathbb R$: $[X_1,X_2]=X_3$, $[X_1,X_3]=X_4$, $[X_2,X_3]=X_5$.}\\[2mm]
By writing the elements  as column vectors using basis $\{X_5, X_4, X_3, X_6, X_2, X_1\}$, we find

$$\mbox{Aut}(\lieg_{5,4}\oplus \mathbb R)=\displaystyle{\left\{\left( \begin{array}{llllll} \beta_2x & \beta_1x & \varepsilon_3 & \varepsilon_6 & \varepsilon_2 &  \varepsilon_1
\\ \alpha_2x & \alpha_1x & \partial_3 &  \partial_6 & \partial_2 &  \partial_1\\ 0 & 0 & x & 0 & \gamma_2 &
\gamma_1\\ 0 & 0 & 0  & \lambda_6 & \lambda_2 & \lambda_1\\ 0 & 0 & 0 & 0 & \beta_2 &  \beta_1
\\ 0 & 0 & 0 & 0 & \alpha_2 &  \alpha_1\end{array}\right)\right\}}<  \mbox{GL}_6(\mathbb R),$$ where $\varepsilon_3=\beta_1\gamma_2-\beta_2\gamma_1,$ $\partial_3=\alpha_1\gamma_2-\alpha_2\gamma_1$, $\lambda_6\ne 0$,
 and $x=\beta_2\alpha_1-\beta_1\alpha_2\ne 0$. Just as in the previous case, there exists an epimorphism of $\mbox{Aut}(\lieg_{5,4}\oplus \mathbb R)$ onto  $\mbox{GL}_2(\mathbb R)\times \R^{*}$ and the kernel is unipotent. Thus, $\Gamma$ is virtually polycyclic.

\medskip

\noindent\underline{{\bf Case 3:} $ \lieg_{4,1}\oplus \mathbb R^2$: $[X_1,X_2]=X_3$, $[X_1,X_3]=X_4$.}\\[2mm]
By writing the elements as column vectors using basis $\{X_4, X_3, X_6, X_5, X_2, X_1\}$, we have

$$\mbox{Aut}(\lieg_{4,1}\oplus \mathbb R^2)=\displaystyle{\left\{\left( \begin{array}{llllll} \alpha^2_1\beta_2 & \alpha_1\gamma_2 & \partial_6 & \partial_5 & \partial_2 & \partial_1
\\ 0 & \alpha_1\beta_2 & 0 &  0 & \gamma_2 &  \gamma_1\\ 0 & 0 & \lambda_6 & \lambda_5 & \lambda_2 &
\lambda_1\\ 0 & 0 & \varepsilon_6  & \varepsilon_5 & \varepsilon_2 & \varepsilon_1\\ 0 & 0 & 0 & 0 & \beta_2 &  \beta_1
\\ 0 & 0 & 0 & 0 & 0 &  \alpha_1\end{array}\right)\right\}}<  \mbox{GL}_6(\mathbb R),$$ where $\alpha_1\beta_2\ne 0$
 and $\lambda_6\varepsilon_5 -\lambda_5\varepsilon_6\ne 0$. There exists an epimorphism  of $\mbox{Aut}(\lieg_{4,1}\oplus \mathbb R^2)$  onto $\mbox{GL}_2(\mathbb R)\times \R^{*}\times \R^{*}$ with unipotent kernel.   It follows that $\Gamma$ is virtually polycyclic.
\end{proof}

Lastly, we study groups which act properly discontinuously  on a contractible algebraic manifold of dimension at most three.

\begin{theorem} \label{thm: app3} {\normalfont Let $\Gamma$ be a finitely generated Zariski dense subgroup of a connected  algebraic group $G$. Let $H$ be a maximal reductive subgroup of $G$. Suppose $G$ acts algebraically on a contractible algebraic manifold $X$ such that there exists $x_0\in X$ with $Hx_0=x_0$. Suppose $\Gamma$ acts properly discontinuously on $X$.
\begin{enumerate}[(a)]
\item If $\mbox{dim}(X)=2$, then  $\Gamma$ is virtually polycyclic. 
\smallskip
\item If  $\mbox{dim}(X)=3$ and $\Gamma$ is not virtually polycyclic, then $\Gamma$ is virtually free, $X$ is $G$-equivariantly isomorphic to $\R^3$, where $G$ acts by Lorentz transformations on $\R^3$.
\end{enumerate}}
\end{theorem}
\begin{proof} Let $\Gamma_1= \Gamma\cap R$ and $G_1$ be the algebraic closure of $\Gamma_1$ in $R$. Denote $\widetilde{\Gamma}=\Gamma/\Gamma_1$ and let $S$ be a Levi factor of $H$. By Selberg's lemma, we can assume that both $\Gamma$ and $\widetilde{\Gamma}$ are torsion-free.

Now, let us consider the orbit space $Gx_0$. It is an algebraic submanifold of $X$ that is $G$-equivariantly isomorphic to the quotient of the unipotent radical of $G$ by a closed connected subgroup. We denote this submanifold by $Y$ and define $\widetilde{Y}=Y/G_1$. Note that when $\Gamma$ acts cocompactly, then $Y=X$ and $\widetilde{Y}=\widetilde{X}$. Proceeding as in the proof of Theorem \ref{thm: main}, it follows that the action of $S$ on the tangent space $T_{\tilde{x}_0}\widetilde{Y}$ of $\widetilde{Y}$ at $\tilde{x}_0=G_1x_0$ is faithful. Then, the representation $\rho:S\hookrightarrow \mbox{GL}(T_{\tilde{x}_0}\widetilde{Y})$ is faithful.


To prove (a), observe that  $\mbox{SL}_2(\R)$ is the only nontrivial connected semisimple subgroup of $\mbox{GL}_2(\R)$. But, $\mbox{SL}_2(\R)$ does not satisfy the assertion of Lemma \ref{lem: lemp11}. Therefore, $S$ must be trivial.

For part (b), we again observe that the groups  $\mbox{SO}(3)$, $\mbox{SO}^0(2,1)$, $\mbox{SL}_2(\R)\times \{1\}$, and $\mbox{SL}_3(\R)$ are the only nontrivial connected semisimple subgroups of $\mbox{GL}_3(\R)$. The case  $S=\mbox{SL}_3(\R)$ is impossible by Lemma \ref{lem: lemp11}.
Suppose that $S=\mbox{SO}(3)$. Then $\widetilde{\Gamma}$ is trivial, because it is a discrete subgroup of a compact group. So, $\Gamma$ is virtually polycyclic.

Next, let us assume $S=\mbox{SL}_2(\R)\times \{1\}$. It follows that $\mbox{dim}(\widetilde{Y})=3$, otherwise the representation $\rho:S\hookrightarrow \mbox{GL}(T_{\tilde{x}_0}\widetilde{Y})$ would violate the eigenvalue one criteria of \ref{lem: lemp11}.
We deduce that $\widetilde{Y}= Y=X=\widetilde{X}$ and $\widetilde{\Gamma}=\Gamma$. It follows that $G$ acts transitively on $X$ and $X$ is $G$-equivariantly isomorphic to the quotient of the unipotent radical $U$ by a closed and connected subgroup $V$ (see \ref{lem: lem6}).  Next, we proceed as in the proof of \ref{lem: lem10}.

First, we can assume that $V$ does not contain a nontrivial normal subgroup of $U$. We denote by $\lieu$ the Lie algebra of $U$ and by $\liev$ the subalgebra corresponding to $V$. We suppose that
$\lieu$ is $c$-step nilpotent and define the subalgebras
$\lieu_i=\gamma_i(\lieu)+ \liev$ forming a descending sequence
\[ \lieu=\lieu_1\supseteq \lieu_2\supseteq \lieu_3 \supseteq \cdots \supseteq \lieu_c \supseteq \liev=\lieu_{c+1}\]
of subalgebras, which are invariant under the action of $H$. In addition, for each $i$, $\lieu_{i+1}$ is an ideal in $\lieu_{i}$. We obtain the following faithful representation of $H$ (see (\ref{mainrep}))
\[\psi: H \rightarrow \GL(\lieu_1/\lieu_2)\times \GL(\lieu_2/\lieu_3) \times \cdots \times \GL(\lieu_c/\lieu_{c+1}).
\]
We distinguish three possible cases. Either $\lieu_c=\lieu$ and $\GL(\lieu_c/\lieu_{c+1})\cong \GL_3(\R)$ or $\lieu_c\ne\lieu$ and $\psi:H\rightarrow \R^*\times \GL_2(\R)$ or  $\lieu_c\ne\lieu$  and $\psi:H\rightarrow \GL_2(\R)\times \R^*$.

In the first case, $\Gamma$ admits an affine action on $\R^3$. So, by applying Theorem 2.1 of \cite{FG}, we conclude that it is virtually polycyclic.

In the second case, there is an $H$-invariant proper ideal  $\liew$ of $\lieu$ such that $\lieu/\liew\cong \R$  and $\liew/\liev\cong \R^2$. Since $\psi$ is faithful and satisfies the eigenvalue one criteria, it is not difficult to see that $H$ acts trivially on $\lieu/\liew$.
Thus, the image of the representation  $\rho:H\hookrightarrow \mbox{GL}(\lieu/\liev)$ lies inside
$$\begin{pmatrix}

\mbox{GL}_2(\R)

& \vrule & \begin{array}{c}

\ast \\

\ast

\end{array}\\

\hline \\[-0.48cm]

0 \hspace{6mm} 0 & \vrule &

1

\end{pmatrix}$$

 Let $W$ be the normal unipotent subgroup of $U$ corresponding to $\liew$. Consider the map $\phi:\Gamma\rightarrow U/W\cong \R$ given by $\gamma\mapsto [\gamma x][x^{-1}]$ where $x\in U/V$ and $[\hspace{1mm}]:U\to U/W$ is the quotient homomorphism. We claim that this map is independent of the choice of $x$ and hence it is a homomorphism. Indeed, since $G$ acts transitively on $U/V$, it acts as in (1). Because $H$ acts trivially on $U/W$, it follows that for $\gamma=uh$, $u\in U$, $h\in H$, we have $[\gamma x][x^{-1}]= [uhxh^{-1}][x^{-1}]=[u][x][x^{-1}]=[u]$. Now, it is routine to check that $\phi$ is a homomorphism. The kernel of $\phi$ is the subgroup of $\Gamma$ that leaves invariant the submanifold $W/V\cong \R^2$. Thus, by part (a), it is virtually polycyclic and so, $\Gamma$ is virtually polycyclic.

In the third and final case, there is an $H$-invariant proper ideal  $\liew$ of $\lieu$ such that $\lieu/\liew\cong \R^2$  and $\liew/\liev\cong \R$. It follows that $H$ acts trivially on $\liew/\liev$. By identifying the tangent space $T_{{x}_0}{X}\cong \lieu/\liev$ with $\R^3$, we obtain that the image of the representation  $\rho:H\hookrightarrow \mbox{GL}(\lieu/\liev)$ is a subgroup of
$$\begin{pmatrix}

1  & \vrule & \ast \hspace{6mm} \ast\\

\hline \\[-0.48cm]

\begin{array}{c}

0\\

0

\end{array}

& \vrule &
\mbox{GL}_2(\R)

\end{pmatrix}$$

Let us suppose that $\Gamma$ is not solvable. By Lemma 2.10 of \cite{FG}, there are elements $\gamma_1$ and $\gamma_2$ in $\Gamma$ such that the projections to $\mbox{GL}_2(\R)$ of their images in $\rho(H)$  are hyperbolic (have real distinct inverse eigenvalues) and share no common eigenspace in $\lieu/\liew$.

The action of $\Gamma$ on $U/V$ descends to the plane $U/W$. It follows that there is a unique point $[u_i]\in U/W$ invariant under the action of $\gamma_i$ for $i=1,2$. This means that the curve $l_i=u_iW/V$ in $U/V$ is invariant under action of $\gamma_i$ where $u_i\in U$ is a preimage of $[u_i]$.
Next, we will argue that these curves coincide. To see this, first we choose a Riemannian metric on the quotient manifold $X/\Gamma$ and endow $X$ with the covering space metric. Let $\Lambda_i=\langle \gamma_i \rangle$  for $i=1,2$ and consider the quotient map $U/V\to U/W\cong \R^2$ which is $\Gamma$-equivariant.
Let $\varepsilon>0$ and denote by $W_i^{\varepsilon}$ the $\varepsilon$-tubular neighborhood of $l_i$ in $U/V$. Since the action of $\Lambda_i$ on $U/V$ leaves $W_i^{\varepsilon}$ invariant, the image of $W_i^{\varepsilon}$ in $U/W$ is a $\Lambda_i$-equivariant neighborhood of the point $[u_i]$.  Since $\gamma_i$ acts as an Anosov diffeomorphism of $U/W$ fixing $[u_i]$, it is not difficult to see that any $\Lambda_i$-equivariant neighborhood of $[u_i]$ in $U/W$ contains the $\Lambda_i$--invariant lines through $[u_i]$ on which $\gamma_i$ acts respectively  as a contracting and expanding map. Let us refer to these two lines as the coordinate axes of $[u_i]$. It follows that the image of $W_i^{\varepsilon}$ in $U/W$ contains the coordinate axes of $[u_i]$. Let $[u_0]$ be an intersection point of the coordinate axes of $[u_1]$ and $[u_2]$. We have just shown that both $W_1^{\varepsilon}$ and $W_2^{\varepsilon}$ contain the curve $u_0W/V$. This implies that every point on  $u_1W/V$ is arbitrarily close to $u_2W/V$ and hence the two curves must coincide.

Now, the group generated by $\gamma_1$ and $\gamma_2$  acts properly discontinuously on the curve $u_1W/V$ and is therefore cyclic. This is a contradiction to the fact that $\tilde{\gamma}_1$ and $\tilde{\gamma}_2$ have no common eigenspace.

Finally, we suppose  $S=\mbox{SO}^0(2,1)$. By Corollary 5.2 in \cite{Abels} (see also 8.24 of \cite{Rag}), $\widetilde{\Gamma}$ is a discrete subgroup of $G/R$ which is locally isomorphic to $S$. Therefore, $\widetilde{\Gamma}$ is a discrete subgroup of isometries of the hyperbolic plane $\mathbb H^2$ and thus $\mbox{cd}(\widetilde{\Gamma})\leq 2$. If $\widetilde{Y}/\widetilde{\Gamma}$ is compact, then  $\widetilde{\Gamma}$ acts crystallographically on $\widetilde{Y}$ and $\mbox{dim}(\widetilde{Y})\leq 2$. So, by part (a), $\widetilde{\Gamma}$ is virtually polycyclic. If $\Gamma$ is not virtually polycyclic, then $\mbox{dim}(\widetilde{Y})=3$. In this particular case, we will not assume that $\Gamma$ is torsion-free. Passing to a quotient of ${\Gamma}$ by a finite subgroup,  we can assume that $G$ acts faithfully on $X$.  We deduce that $\widetilde{Y}= Y=X=\widetilde{X}$ and $\widetilde{\Gamma}$ is a quotient of $\Gamma$ by a finite subgroup. It follows that $G$ acts transitively on $X$ and $X$ is $G$-equivariantly isomorphic to the quotient of $U$ by a closed and connected subgroup $V$. We will argue that $V$ is a normal subgroup of $U$ such that $U/V$ is abelian.  As before, we can assume that $V$ does not contain a nontrivial normal subgroup of $U$ and construct a faithful representation $\psi$ of $H$.
It is not difficult to see that, since $\mbox{SO}^0(2,1)\leq H$, the representation is only  faithful when $\lieu_c=\lieu$ and $\GL(\lieu_c/\lieu_{c+1})\cong \GL_3(\R)$. This finishes our claim.

Since there are no proper  connected algebraic subgroups of $\mbox{GL}_3(\R)$ properly containing $\mbox{SO}(2,1)$ such that all their elements have an eigenvalue equal to one, it follows that $H=\mbox{SO}(2,1)$. Hence, $G \cong U\rtimes \mbox{SO}(2,1)$ and $G$ acts on $ U/V\cong \R^3$ by Lorentz transformations.

Now, according to a theorem of Mess (see \cite{Mess}), $\widetilde{\Gamma}$ cannot contain the fundamental group of a closed surface. This implies that $\mbox{vcd}(\widetilde{\Gamma})=1$ an so,  by Stallings's theorem,  $\widetilde{\Gamma}$ is virtually free. Then, ${\Gamma}$ is finite by virtually free and it is also virtually torsion-free. Therefore, $\Gamma$ is virtually free.\end{proof}

Immediately from the theorem, we obtain the generalized Auslander conjecture (see Question 1) up to dimension three.

\begin{corollary}\label{general_Auslander_dim3}{\normalfont Let $G$ be an algebraic group acting algebraically on a contractible algebraic manifold $X$ of dimension at most three such that a maximal reductive subgroup of $G$ fixes a point of $X$. If $\Gamma$ acts crystallographically on $X$ via a representation $\rho:\Gamma\rightarrow G$, then $\Gamma$ virtually polycyclic.}
\end{corollary}
\begin{center}\textbf{Acknowledgements}\end{center}

We are grateful to Fritz Grunewald for his correspondence during an earlier version of the paper. We also thank Oliver Baues for the useful remarks on Lemma \ref{lem: lem8}. Finally, we thank the  referee for  the helpful comments.

\end{document}